\documentclass[leqno,letterpaper,12pt]{scrartcl}
\usepackage{amsmath,upref,amssymb,amsthm,amsxtra,exscale}
\usepackage{cite}
\usepackage{enumitem}
\usepackage[T1]{fontenc}
\usepackage{lmodern}
\usepackage[utf8]{inputenc}
\usepackage{mathtools}
\usepackage{tikz-cd}
\usepackage[reftex]{theoremref}
\usetikzlibrary{external}
\newcommand{\missing}{\ensuremath{\clubsuit\clubsuit\clubsuit}}

\numberwithin{equation}{section}

\theoremstyle{plain}
\newtheorem{theorem}{Theorem}[section]
\newtheorem{lemma}[theorem]{Lemma}

\newtheorem{corollary}[theorem]{Corollary}

\theoremstyle{definition}

\theoremstyle{remark}

\newtheorem*{acknowledgement}{Acknowledgements}

\newcommand{\dN}{\mathbb{N}}

\newcommand{\dR}{\mathbb{R}}

\newcommand{\dZ}{\mathbb{Z}}

\newcommand{\cF}{\mathcal{F}}

\newcommand{\cL}{\mathcal{L}}

\newcommand{\cN}{\mathcal{N}}

\newcommand{\cR}{\mathcal{R}}

\newcommand{\rmb}{\mathrm{b}}

\newcommand{\rmd}{\mathrm{d}}
\newcommand{\rme}{\mathrm{e}}

\newcommand{\olB}{\overline{B}}

\DeclareMathOperator{\dist}{dist}

\newcommand{\dint}{\,\rmd}
\newcommand{\ssm}{\backslash}

\newcommand{\lr}[3]{#1#3#2}
\newcommand{\xlr}[3]{\left#1#3\right#2}
\newcommand{\biglr}[3]{\bigl#1#3\bigr#2}

\newcommand{\bigglr}[3]{\biggl#1#3\biggr#2}

\newcommand{\abs}[1]{\lr\lvert\rvert{#1}}
\newcommand{\xabs}[1]{\xlr\lvert\rvert{#1}}

\newcommand{\norm}[1]{\lr\lVert\rVert{#1}}
\newcommand{\xnorm}[1]{\xlr\lVert\rVert{#1}}

\newcommand{\fracwithdelims}[4]{\genfrac{#1}{#2}{}{}{#3}{#4}}
\newcommand{\coloneqq}{:=}

\allowdisplaybreaks

\begin{document}

\title{Precise exponential decay for solutions of semilinear elliptic
  equations and its effect on the structure of the solution set for a
  real analytic nonlinearity} 

\author{Nils  Ackermann\thanks{Supported by CONACYT grant \missing{} and PAPIIT grant
    IN104315 (Mexico)}\and Norman Dancer}
\date{}
\maketitle
\begin{abstract}
  We are concerned with the properties of weak solutions of the
  stationary Schrödinger equation $-\Delta u + Vu = f(u)$, $u\in
  H^1(\mathbb{R}^N)\cap L^\infty(\mathbb{R}^N)$, where $V$ is Hölder
  continuous and $\inf V>0$.  Assuming $f$ to be continuous and
  bounded near $0$ by a power function with exponent larger than $1$
  we provide precise decay estimates at infinity for solutions in
  terms of Green's function of the Schrödinger operator.  In some
  cases this improves known theorems on the decay of solutions.  If
  $f$ is also real analytic on $(0,\infty)$ we obtain that the set of
  positive solutions is locally path connected.  For a periodic
  potential $V$ this implies that the standard variational functional
  has discrete critical values in the low energy range and that a
  compact isolated set of positive solutions exists, under additional
  assumptions.
\end{abstract}

\section{Introduction}
\label{sec:introduction}

We are interested in the properties of weak solutions of
\begin{equation}
  \label{eq:12}\tag{P}
    -\Delta u + Vu = f(u),\qquad u\in H^1(\dR^N)\cap L^\infty(\dR^N),
\end{equation}
where $f$ is continuous, $f(u)\le C\abs{u}^q$ near $0$, for some
$q>1$, $V$ is Hölder continuous, bounded, and $\mu_0\coloneqq \inf
V>0$.

In the first part of this work we consider exponential decay of
solutions of \eqref{eq:12}.  We say that a function $u$ \emph{decays
  exponentially at infinity with exponent $\nu>0$ if
  $\limsup_{\abs{x}\to\infty}\rme^{\nu\abs{x}}u(x)<\infty$}.

One of the most thorough studies of this question is an article by
Rabier and Stuart~\cite{MR2001f:35139}, where general quasilinear
equations are considered.  We give a more precise description of the
decay of such solutions $u$ in terms of Green's function $G$ of the
Schrödinger operator $-\Delta+V$.  Setting $H(x)\coloneqq G(x,0)$ we
show that $u$ is bounded above by a multiple of $H$ near infinity.  In
particular, $u$ decays as fast as $H$.  In some cases this improves
the estimates obtained in \cite{MR2001f:35139}.  To illustrate this,
suppose for a moment that $V$ is a positive constant $\mu_0$.  Since
$H$ decays exponentially at infinity with exponent $\sqrt{\mu_0}$, our
result yields the same for every solution of \eqref{eq:12}, while
\cite{MR2001f:35139} only yields exponential decay at infinity with
exponent $\nu$ for every $\nu\in(0,\sqrt{\mu_0})$.  Their method could be
extended to yield the same result only if $f(u)/u\le0$ near $0$.

On the other hand, if $u$ is a \emph{positive} solution of
\eqref{eq:12} then we obtain that $u$ is bounded below by a multiple
of $H$, that is to say, the decay of $u$ and $H$ are
\emph{comparable}.  We are not aware of a similar result in the
literature.

These comparison results are a consequence of \emph{a priori}
exponential decay of every solution of \eqref{eq:12}, of the behavior
of $f$ near $0$ and of a deep result of Ancona \cite{MR1482989} about
the comparison of Green's functions for positive Schrödinger operators
whose potentials only differ by a function that decays sufficiently
fast at infinity.

In the second part of our paper we assume in addition that $f$ is a
real analytic function, either on all of $\dR$ or solely on
$(0,\infty)$.  In the complete text analyticity is always \emph{real
  analyticity}.  We have used analyticity before to obtain results on
the path connectivity of bifurcation branches and solutions sets
\cite{MR1962054,MR0322615,MR0375019}.  Set $F(u):=\int_0^uf$
and introduce the variational functional
\begin{equation*}
  J(u):=\frac12\int_{\dR^N}(\abs{\nabla u}^2+Vu^2)-\int_{\dR^N}F(u)
\end{equation*}
for weakly differentiable functions $u\colon\dR^N\to\dR$ such that the
integrals are well defined.  If $K$ is the set of solutions of
\eqref{eq:12} and $K_+$ the set of \emph{positive} solutions of
\eqref{eq:12} then we show that the analyticity of $f$ implies local
path connectedness of $K$ in the first case and of $K_+$ in the second
case.  Moreover, it follows that $J$ is locally constant on $K$,
respectively $K_+$.  We achieve this by working in spaces of
continuous functions with norms weighted at infinity by powers of $H$.
As a consequence, the set $K_+$ lies in the interior of the positive
cone of a related weighted space.  This allows to transfer the
analyticity from $f$ to the set $K_+$ in the case where $f$ is only
analytic in $(0,\infty)$.  From the analyticity of a set its local
path connectedness follows from a classical triangulation theorem
\cite{MR0173265,MR0159346}.

In the last part we apply these results to a special case of
\eqref{eq:12}, where we assume $V$ to be periodic in the coordinates.
Set $c_0\coloneqq\inf J(K)>0$, the ground state energy.  Under
additional growth assumptions on $f$ we obtain that $J(K)$,
respectively $J(K_+)$, has no accumulation point in the so called
\emph{low energy range} $[c_0,2c_0)$.  If in addition $V$ is
reflection symmetric and $f$ satisfies an Ambrosetti-Rabinowitz-like
condition, an earlier separation Theorem of ours \cite{MR2488693}
yields, together with the aforementioned conclusion, the existence of
a compact set $\Lambda$ of positive solutions at the ground state
energy that is isolated in the set of solutions $K$.

The latter result is of interest when one considers the existence of
so-called \emph{multibump solutions}, which are nonlinear
superpositions of translates of solutions in the case of a periodic
potential $V$.  It is to be expected that such a set $\Lambda$ can be
used as a base for nonlinear superposition.  This would yield a much
weaker condition than that imposed in the seminal article
\cite{MR93k:35087} and its follow-up works, where the existence of a
\emph{single} isolated solution was required.

The present article is structured as follows: In
section~\ref{sec:comp-solut} we study the exact decay of solutions at
infinity in terms of Green's function of the Schrödinger operator.
Section~\ref{sec:real-analyticity} is devoted to the consequences of
analyticity of $f$.  And last but not least,
Section~\ref{sec:appl-peri-potent} treats the consequences for the
solution set of \eqref{eq:12} if the potential $V$ is periodic.

\subsection{Notation}
\label{sec:notation}

For a metric space $(X,d)$, $r>0$, and $x\in X$ we denote
\begin{align*}
  B_r(x;X)&\coloneqq\{y\in X\mid d(x,y)< r\},\\
  \olB_r(x;X)&\coloneqq\{y\in X\mid d(x,y)\le r\},\\
  S_r(x;X)&\coloneqq\{y\in X\mid d(x,y)= r\}.
\end{align*}
We also set $B_rX\coloneqq B_r(0;X)$ if $X$ is a normed space and use
analogous notation for the closed ball and the sphere.  If $X$ is
clear from context we may omit it in the notation.  For $k\in\dN_0$
denote by $C^k_\rmb(\dR^N)$ the space of real valued functions of
class $C^k$ on $\dR^N$ such that all derivatives up to order $k$ are
bounded.  We set $C_\rmb(\dR^N)\coloneqq C^0_\rmb(\dR^N)$.

\section{Exact Decay of Solutions}
\label{sec:comp-solut}

This section is concerned with comparing the decay of a solution to
\eqref{eq:12} with Green's function of the Schrödinger operator
$T\coloneqq -\Delta+V$.  We show that if the nonlinearity $f$ is well
behaved at $0$ then a solution decays at least as fast as Green's
function.  If in addition the solution is positive then it decays at
most as fast as Green's function.

Suppose that $N\in\dN$.  The principal regularity and positivity
requirements for the potential we use are contained in the following
condition:
\begin{enumerate}[label=\textup{(V\arabic*)},series=vcond]
\item \label{item:1} $V\colon\dR\to\dR$ is Hölder continuous and
  bounded, and $\mu_0\coloneqq\inf V>0$.
\end{enumerate}

We will need to know \emph{a priori} that weak solutions of
\eqref{eq:12} and related problems decay exponentially at infinity.
For easier reference we include a pertinent result here, even though
this fact is in principle well known.
\begin{lemma}
  \th\label{lem:a-priori-decay} Assume \ref{item:1}.  Suppose that $f\in
  C(\dR)$ satisfies $f(u)=o(u)$ as $u\to0$ and that $v\in
  L^\infty(\dR^N)$ decays exponentially at infinity.  If either $u\in
  H^1(\dR^N)\cap L^\infty(\dR^N)$ is a weak solution of $-\Delta u+
  Vu=f(u)$ or $u\in H^1(\dR^N)$ is a weak solution of $-\Delta u+
  Vu=v$ then $u$ is continuous and decays exponentially at infinity.
\end{lemma}
\begin{proof}
  In the first case we may alter $f$ outside of the range of $u$ in
  any way we like.  Therefore \cite[Lemma~5.3]{MR2151860} applies and
  yields, together with standard regularity theory and bootstrap
  arguments using \emph{a priori} estimates (e.g., \cite{MR737190},
  Theorem~9.11 and Lemma~9.16), that $u$ is continuous and decays
  exponentially at infinity.

  For the second case suppose that $\abs{v(x)}\le
  C_1\rme^{-C_2\abs{x}}$ for all $x\in\dR^N$, with constants
  $C_1,C_2>0$.  For $r>0$ denote
  \begin{equation*}
    Q(r)\coloneqq\int_{\dR^N\ssm \olB_r}(\abs{\nabla u}^2+Vu^2).
  \end{equation*}
  We claim that $Q(r)$ decays exponentially at infinity.  By
  contradiction we assume that this were not the case.  Then
  \begin{equation}\label{eq:13}
    \inf_{r\ge0}\rme^{C_2r}Q(r)>0.
  \end{equation}
  For $r\ge0$ define the cutoff function $\zeta_r$ as in the proof of
  \cite[Lemma~5.3]{MR2151860} and set $u_r(x)\coloneqq
  \zeta(\abs{x}-r)u(x)$.  Let $\delta\coloneqq\mu_0$.  It follows from
  Hölder's inequality and \eqref{eq:13} that
  \begin{multline*}
    \xabs{\int_{\dR^N}\nabla u\nabla u_r+Vuu_r}
    =\xabs{\int_{\dR^N}vu_r}
    \le C_1\int_{\dR^N\ssm \olB_r}\rme^{-C_2\abs{x}}\abs{u(x)}\dint x\\
    \le C \sqrt{Q(r)}\rme^{-C_2r}
    \le C Q(r)\rme^{-C_2r/2}
    \le \frac{\delta}{2} Q(r)
  \end{multline*}
  for $r$ large enough.  This replaces Equation~5.3 of
  \cite{MR2151860}.  As in that proof it follows that
  \begin{equation*}
    \frac{Q(r+1)}{Q(r)}\le \frac{1+\delta}{1+2\delta}<1
  \end{equation*}
  for large $r$, so $Q(r)$ decays exponentially at infinity.  
  Again using standard  regularity estimates we obtain that $u$ is
  continuous and decays exponentially at infinity.
\end{proof}

By \cite[Theorem~4.3.3(iii)]{MR1326606} the operator $T$ is
subcritical, according to the definition in Sect.~4.3 \emph{loc.\
  cit.}  Hence $T$ possesses a Green's function $G(x,y)$, i.e., a
function that satisfies
\begin{equation*}
  TG(x,y)=\delta(x-y).
\end{equation*}
Moreover, $G$ is positive.  Denote $H(x):=G(x,0)$ for $x\neq0$.  We
collect some properties of $H$ needed later on:
\begin{lemma}\th\label{lem:properties-h}
  The function $H\colon\dR^N\ssm\{0\}\to\dR$ satisfies:
  \begin{enumerate}[label=\textup{(\alph*)}]
  \item \label{item:5} $TH\equiv0$;
  \item \label{item:6} $H\in C^2(\dR^N\ssm\{0\})$;
  \item \label{item:9} $H>0$;
  \item \label{item:7} $\liminf_{x\to0}H(x)>0$;
  \item \label{item:8} $\limsup_{\abs{x}\to\infty}
    \rme^{\sqrt{\mu_0}\, \abs{x}}H(x)<\infty$.
  \end{enumerate}
\end{lemma}
\begin{proof}
  \ref{item:5} and \ref{item:6} are proved in
  \cite[Theorem~4.2.5(iii)]{MR1326606}, \ref{item:9} is a consequence
  of $G>0$, and \ref{item:7} is given by
  \cite[Theorem~4.2.8]{MR1326606}.

  In order to prove \ref{item:8}, consider the function
  $\psi\colon\dR^N\to\dR$ given by
  $\psi(x)\coloneqq\rme^{-\sqrt{\mu_0}\,\abs{x}}$.  Then $\psi$ is a
  supersolution for $T$ on $\dR^N\ssm\{0\}$.  Take $\alpha>0$ large
  enough such that $\alpha\psi\ge H$ on $S_1$.  Denote Green's
  function for $T$ on $B_k$ with Dirichlet boundary conditions by
  $\widetilde{G}_k$, for $k\in\dN$, and set
  $\widetilde{H}_k\coloneqq\widetilde{G}_k(\cdot,0)$.  Then
  $T\widetilde{H}_k\equiv0$ on $B_k\ssm\{0\}$ and $\lim_{\abs{x}\to
    k}\widetilde{H}_k(x)=0$, by \cite[Theorem~7.3.2]{MR1326606}.
  Moreover, \cite[Theorem~4.3.7]{MR1326606} implies that
  $\widetilde{H}_k(x)\to H(x)$ as $k\to\infty$, and
  $(\widetilde{H}_k)$ is an increasing sequence.  It follows that
  $\widetilde{H}_k\le\alpha\psi$ on $S_1$ and hence, by the maximum
  principle, that $\widetilde{H}_k\le\alpha\psi$ in $\olB_k\ssm B_1$
  for all $k$.  Therefore, $H\le\alpha\psi$ on $\dR^N\ssm B_1$ and the
  claim follows.
\end{proof}

We now state the main result of this section:

\begin{theorem}
  \th\label{thm:decay-at-infty} 
  \begin{enumerate}[label=\textup{(\alph*)}]
  \item \label{item:12} Suppose that $w\in L^\infty(\dR^N)$ satisfies
    \begin{equation*}
      \abs{w(x)}\le C_1\rme^{-C_2\abs{x}}
    \end{equation*}
    for $x\in\dR^N$, with some fixed $C_1,C_2>0$.  If $u\in H^1(\dR^N)$ is a weak
    solution of
    \begin{equation*}
      -\Delta u+(V-w)u=0
    \end{equation*}
    then there exists, for every $\delta>0$, some $R_0>0$, depending
    only on $\delta$, $N$, $\inf V$, $\norm{V}_{\infty}$, $C_1$ and
    $C_2$, such that for every $R\ge R_0$ 
    \begin{equation}\label{eq:18}
      \limsup_{\abs{x}\to\infty}\frac{\abs{u(x)}}{H(x)}
      \le(1+\delta)^2 \max_{x\in S_R}\frac{\abs{u(x)}}{H(x)}.
    \end{equation}
    In particular,
    \begin{equation}\label{eq:14}
      \limsup_{\abs{x}\to\infty} \rme^{\sqrt{\mu_0}\,\abs{x}}\abs{u(x)}<\infty.
    \end{equation}
  \item \label{item:13} If in addition to the hypotheses of
    \ref{item:12} $u$ is positive then there exists, for every
    $\delta>0$, some $R_0>0$, depending only on $\delta$, $N$, $\inf V$,
    $\norm{V}_{\infty}$, $C_1$ and $C_2$, such that for every $R\ge R_0$
    \begin{equation}\label{eq:19}
      \liminf_{\abs{x}\to\infty}\frac{u(x)}{H(x)}
      \ge(1+\delta)^{-2} \min_{x\in S_R}\frac{u(x)}{H(x)}.
    \end{equation}
  \item \label{item:14} If $v\in L^\infty(\dR^N)$ satisfies that $v/H$
    decays exponentially at $\infty$ and if $u\in H^1(\dR^N)$ is a
    weak solution of
    \begin{equation*}
      -\Delta u+ Vu=v
    \end{equation*}
    then there exist continuous functions $u_1$ and $u_2$ such that
    $u=u_1-u_2$, $Tu_1=v^+$, $Tu_2=v^-$, and such that for each
    $i=1,2$ either $u_i\equiv0$, or $u_i>0$ and
    \begin{equation}
      \label{eq:21}
      0
      <\liminf_{\abs{x}\to\infty}\frac{u_i(x)}{H(x)}
      \le \limsup_{\abs{x}\to\infty}\frac{u_i(x)}{H(x)}
      <\infty.
    \end{equation}
    In particular,
    \begin{equation*}
      \limsup_{\abs{x}\to\infty}\frac{\abs{u(x)}}{H(x)}
      <\infty.
    \end{equation*}
  \end{enumerate}
\end{theorem}
\begin{proof}
  \noindent\textbf{\ref{item:12}} Standard \emph{a priori} estimates, as
  mentioned in the proof of \th\ref{lem:a-priori-decay}, yield that
  $u\in L^\infty(\dR^N)$.  Hence also $wu$ has exponential decay at
  infinity and \th\ref{lem:a-priori-decay} yields in particular that
  \begin{equation}
    \label{eq:23}
    u(x)\to0\qquad\text{as }\abs{x}\to\infty.
  \end{equation}

  We take $R>1$ large enough such that
  \begin{equation*}
    \sup\abs{w}\le \varepsilon_0\coloneqq\frac{\mu_0}{2}
  \end{equation*}
  in $\dR^N\ssm \olB_{R-1}$ and define $\eta\colon\dR^N\to[0,1]$ by
  \begin{equation*}
    \eta(x)\coloneqq
    \begin{cases}
      0,&\qquad \abs{x}\le R-1,\\
      \abs{x}-R+1,&\qquad R-1\le\abs{x}\le R,\\
      1,&\qquad \abs{x}\ge R.
    \end{cases}
  \end{equation*}
  Then $\inf (V-\eta w)\ge\varepsilon_0>0$.  Hence also $T_1\coloneqq
  -\Delta+(V-\eta w)$ is subcritical on $\dR^N$ and possesses a
  positive Green's function $G_1$.  Since we are not assuming $w$ to
  be locally Hölder continuous, here we refer to \cite{MR890161} and
  \cite{MR874676} for the existence of the positive Green's function.
  Set $H_1(x):=G_1(x,0)$ for $x\neq0$.  In the notation of
  \cite{MR1482989} use our $\varepsilon_0$ and set $r_0:=1/4$,
  $c_0\coloneqq1$, and $p:=2N$.  Note that the bottom of the spectrum
  of $T$ and $T_1$ as operators in $L^2$ with domain $H^2$ is greater
  than or equal to $\varepsilon_0$.  Denote
  \begin{equation*}
    \widetilde{C}:=\sup\biglr\{\}{\,\norm{v}_{L^N(\olB_{r_0})}\bigm| 
      v\in L^\infty(\olB_{r_0}),\ \norm{v}_{L^\infty(\olB_{r_0})}=1\,}
  \end{equation*}
  and set $\theta:=1+\widetilde{C}(C_1+\norm{V}_\infty)$.  Define the decreasing
  function
  \begin{equation*}
    \Psi_R(s):=
    \begin{cases}
      C_1\rme^{-C_2(R-1)}&\qquad 0\le s\le R\\
      C_1\rme^{-C_2(s-1)}&\qquad s\ge R
    \end{cases}
  \end{equation*}
  so $\norm{\eta w} _{L^\infty(\olB_{r_0}(y))} \le
  \Psi_R(\abs{y})$ for $y\in\dR^N$.  Using these constants, the
  function $\Psi_R$ and the fact that
  \begin{equation*}
    \lim_{R\to\infty}\int_0^\infty\Psi_R=0,
  \end{equation*}
  \cite[Theorem~1]{MR1482989} yields
  \begin{equation}
    \label{eq:2}
    \frac1{1+\delta} H(x)\le H_1(x)\le (1+\delta)H(x)
  \end{equation}
  for $\abs{x}\ge r_0$ if $R$ is chosen large enough, only depending
  on $\delta$, $N$, $\inf(V)$, $\norm{V}_\infty$, $C_1$ and $C_2$.

  The function $H_1$ is continuous in $\dR^N\ssm\{0\}$ and satisfies
  $T_1H_1\equiv0$ in $\dR^N\ssm\{0\}$ in the weak sense.  Moreover,
  $T_1u\equiv 0$ on $\dR^N\ssm \olB_R$ in the weak sense.  Set
  \begin{equation*}
    C_3\coloneqq(1+\delta)^2 \max_{x\in S_R}\frac{\abs{u(x)}}{H(x)}
  \end{equation*}
  Then we have by \eqref{eq:2}
  \begin{equation*}
    \abs{u}
    \le\frac{C_3}{(1+\delta)^2}H
    \le \frac{C_3}{1+\delta}H_1
    \qquad\text{on } S_R.
  \end{equation*}
  Note that $H_1(x)\to0$ as $\abs{x}\to\infty$, by
  \th\ref{lem:properties-h}\ref{item:8} and \eqref{eq:2}.  Hence
  \eqref{eq:23}, the maximum principle for weak
  supersolutions \cite[Theorem~8.1]{MR737190} and again \eqref{eq:2} yield
  \begin{equation*}
    \abs{u}\le\frac{C_3}{1+\delta}H_1\le C_3 H
    \qquad\text{on }\dR\ssm B_R,
  \end{equation*}
  that is, \eqref{eq:18}.  Together with
  \th\ref{lem:properties-h}\ref{item:8} we obtain \eqref{eq:14}.

  \noindent\textbf{\ref{item:13}} Define
  \begin{equation*}
    C_4:=(1+\delta)^2 \max_{x\in S_R}\frac{H(x)}{u(x)}.
  \end{equation*}
  Then \eqref{eq:2} implies that
  \begin{equation*}
    H_1\le(1+\delta)H\le\frac{C_4}{1+\delta}u
    \qquad\text{on }S_R.
  \end{equation*}
  The maximum principle  yields
  \begin{equation*}
    H_1\le\frac{C_4}{1+\delta}u
    \qquad\text{on }\dR\ssm B_R,
  \end{equation*}
  so \eqref{eq:2} implies \eqref{eq:19}.

  \noindent\textbf{\ref{item:14}} The operator $T\colon
  H^2(\dR^N)\to L^2(\dR^N)$ has a bounded inverse by \ref{item:1}.
  Denote $v^+\coloneqq\max\{0,v\}$ and set $v^-\coloneqq v^+-v$.
  Define $u_1\coloneqq T^{-1}v^+\in H^1(\dR^N)$ and $u_2\coloneqq
  T^{-1}v^-\in H^1(\dR^N)$.  Again we find by
  \th\ref{lem:a-priori-decay} that
  \begin{equation*}
    u_i(x)\to0\qquad\text{as }\abs{x}\to\infty,\ i=1,2.
  \end{equation*}
  If $u_1$ is not the zero function then it is positive, by the strong
  maximum principle.  Using
  \begin{equation*}
    \left.\begin{aligned}
      Tu_1&\ge 0\\
      TH&=0
    \end{aligned}\quad\right\}
  \qquad\text{in }\dR^N\ssm\{0\}
  \end{equation*}
  the maximum principle yields
  \begin{equation*}
          0
      <\liminf_{\abs{x}\to\infty}\frac{u_1(x)}{H(x)}.
  \end{equation*}
  Hence also $v^+/u_1$ decays exponentially at infinity, and
  \ref{item:12} implies that
  \begin{equation*}
    \limsup_{\abs{x}\to\infty}\frac{u_1(x)}{H(x)}
      <\infty.
  \end{equation*}
  This yields \eqref{eq:21} for $i=1$.  The case $i=2$ follows
  analogously.
\end{proof}

For the semilinear problem \eqref{eq:12} we obtain:

\begin{corollary}\th\label{cor:compare-solutions}
  Assume \ref{item:1}.  Suppose that $f\colon\dR\to\dR$ is continuous
  and that there are $C,M>0$ and $q>1$ such that $\abs{f(u)}\le
  C\abs{u}^q$ for $\abs{u}\le M$.  If $u$ is a weak solution of
  \eqref{eq:12} then $u$ has the properties claimed in
  \th\ref{thm:decay-at-infty}, \ref{item:12} and \ref{item:14}.  If
  in addition $u$ is positive, then $u$ has the property claimed in
  \th\ref{thm:decay-at-infty}\ref{item:13}.
\end{corollary}
\begin{proof}
  By our hypotheses on $f$ \th\ref{lem:a-priori-decay} implies
  exponential decay of $u$ at infinity.  Hence also $w\coloneqq
  f(u)/u$ decays exponentially at infinity.  Since $u$ is a solution
  of $-\Delta u+(V-w)u=0$ \th\ref{thm:decay-at-infty}\ref{item:12}
  applies.  Therefore also $f(u)/H$ has exponential decay at
  infinity.  These facts yield the claims.
\end{proof}

\section{Real Analyticity}
\label{sec:real-analyticity}

Using the precise decay results of the previous section we construct a
weighted space $Y$ of continuous functions that contains all solutions
of \eqref{eq:12} and is such that the positive solutions are contained
in the interior of the positive cone of $Y$.  Assuming analyticity of
the nonlinearity (on $(0,\infty)$) with appropriate growth bounds we
obtain a setting where the (positive) solution set is locally a finite
dimensional analytic set and hence locally path connected.

Denote $2^*\coloneqq \infty$ if $N=1$ or $2$, $2^*\coloneqq 2N/(N-2)$
if $N\ge3$ and consider the following conditions on $f$:

\begin{enumerate}[label=\textup{(F\arabic*)},series=fcond]
\item \label{item:2} $f\in C^1(\dR)$, $f(0)=f'(0)=0$;
\item \label{item:4} $f$ is analytic in $\dR$ and for every $M>0$
  there are numbers $a_k\in\dR$ ($k\in\dN_0$) such that
  \begin{equation*}
    \limsup_{k\to\infty}\frac{a_k}{k!}<\infty
  \end{equation*}
  and
  \begin{equation*}
    \abs{f^{(k)}(u)}\le a_k\abs{u}^{\max\{0,2-k\}}
  \end{equation*}
  for $\abs{u}\le M$ and $k\in\dN_0$.
\item \label{item:3} $f$ is analytic in $\dR^+$ and for every
  $M>0$ there are numbers $p\in(1,2^*-1)$ and $a_k\in\dR$
  ($k\in\dN_0$) such that
  \begin{equation*}
    \limsup_{k\to\infty}\frac{a_k}{k!}<\infty
  \end{equation*}
  and
  \begin{equation*}
    \abs{f^{(k)}(u)}\le a_k\abs{u}^{p-k}
  \end{equation*}
  for $u\in(0,M]$ and $k\in\dN_0$; in this case we are only interested in
  positive solutions of \eqref{eq:12} and may take $f$ to be odd, for
  notational convenience.
\item \label{item:20} There are $C>0$ and $\tilde q\in(1,2^*-1)$ such
  that $\abs{f(u)}\le C(1+\abs{u}^{\tilde q})$ for all $u\in\dR$.
\end{enumerate}

To give a trivial example of a function satisfying these conditions,
take $p$ as in condition \ref{item:3}.  Then
$f(u)\coloneqq\abs{u}^{p-1}u$ satisfies conditions \ref{item:2},
\ref{item:3} and \ref{item:20}.

If either \ref{item:4} or \ref{item:3} holds true, then there is $q>1$
such that for every $M>0$ there are $a_0,a_1\in\dR$ such that
\begin{equation}
  \label{eq:4}
  \abs{f(u)}\le a_0\abs{u}^q\quad\text{and}\quad
  \abs{f'(u)}\le a_1\abs{u}^{q-1}\qquad \text{if }\abs{u}\le M.
\end{equation}
To see this take $q\coloneqq2$ if \ref{item:4} holds true, take
$q\coloneqq p$ if \ref{item:3} holds true, and use the respective
numbers $a_0$ and $a_1$ given for $M$ by these hypotheses.

Denote by $K$ the set of non-zero solutions of \eqref{eq:12} and set
$K_+\coloneqq\{u\in K\mid u\ge 0\}$.  Denote by $\cF$ the superposition
operator induced by $f$.  Then every $u\in K$ satisfies $Tu=\cF(u)$.
Our goal is to produce a Banach space $Y$ such that
\begin{equation*}
  \begin{aligned}
    \Gamma\colon Y&\to Y\\ u&\mapsto u- T^{-1}\cF(u)
  \end{aligned}
\end{equation*}
is well defined and such that $K\subseteq Y$ is the zero set of
$\Gamma$.  Moreover, we need $\Gamma$ to be a Fredholm map, analytic
in a neighborhood of $K$ if \ref{item:4} holds true, and analytic in a
neighborhood of $K_+$ if \ref{item:3} holds true.  In the latter case,
because $f$ is not analytic at $0$ we need that $K_+$ belongs to the
interior of the positive cone of $Y$.

Consider the function $H$ defined in Section~\ref{sec:comp-solut}.
Pick a number $b_0\in(0,\infty)$ such that
$b_0\le\liminf_{x\to0}H(x)$.  By \th\ref{lem:properties-h} the
function $\varphi\colon\dR^N\to\dR^N$ defined by
\begin{equation*}
  \varphi(x)\coloneqq\min\{b_0,H(x)\}
\end{equation*}
is continuous, positive, and has the same decay at infinity as $H$.
Define the spaces
\begin{equation*}
  X_\alpha:=\bigglr\{\}{\,u\in C(\dR^N)\biggm| \norm{u}_{X_\alpha}
    := \sup_{x\in\dR^N}\xabs{\frac{u(x)}{\varphi(x)^\alpha}}<\infty\,}
\end{equation*}
for $\alpha>0$.  Together with its weighted norm
$\norm{\,\cdot\,}_{X_\alpha}$, $X_\alpha$ is a Banach space.  Set
\begin{equation*}
  Y:=X_1\cap C^1_{\rmb}(\dR^N)\cap H^1(\dR^N)
\end{equation*}
and $\norm{\cdot}_Y:=\norm{\cdot}_{X_1} + \norm{\cdot}_{C^1_{\rmb}} +
\norm{\cdot}_{H^1}$.  By \eqref{eq:4} and
\th\ref{cor:compare-solutions} $K\subseteq Y$.

We prove the basic properties of the space $Y$ and related mapping
properties of the maps $T$ and $\cF$:

\begin{lemma}
  \th\label{lem:properties-y} Suppose that \ref{item:1}, \ref{item:2}
  and one of \ref{item:4} or \ref{item:3} are satisfied.  Then the
  following hold true:
  \begin{enumerate}[label=\textup{(\alph*)}]
  \item \label{item:15} $T^{-1}\colon X_\alpha\to Y$ is well defined
    and continuous if $\alpha>1$.
  \item \label{item:16} Let $q$ be given by \eqref{eq:4}.  If
    $\alpha<\min\{2,q\}$ then $\cF(Y)\subseteq X_\alpha$, and
    $\cF\colon Y\to X_\alpha$ is completely continuous, i.e., it is
    continuous and maps bounded sets into relatively compact sets.
    Moreover, it is continuously differentiable in $Y$.
  \item \label{item:23} The set $K_+$ is contained in the interior of
    the positive cone of $Y$.
  \item \label{item:17} If \ref{item:20} is satisfied then on $K$ the
    $H^1$-topology and the $Y$-topology coincide.
  \end{enumerate}
\end{lemma}
\begin{proof}
  \textbf{\ref{item:15}:} For any $s\ge2$ the linear mapping
  $T^{-1}\colon L^s(\dR^N)\to W^{2,s}(\dR^N)$ is well defined and
  continuous because of \ref{item:1}.  If $v\in X_\alpha\subseteq
  L^s(\Omega)$ then by the definition of $X_\alpha$ and by
  \th\ref{lem:properties-h}\ref{item:8} the function $v/H$ decays
  exponentially at infinity.  For $u\coloneqq T^{-1}v$ it follows from
  \th\ref{thm:decay-at-infty}\ref{item:14} that $u\in X_1$.
  Therefore
  \begin{equation*}
    \begin{tikzpicture}[commutative diagrams/every diagram]
      \matrix[matrix of math nodes, name=m, commutative diagrams/every cell] {
        X_1 & L^s \\
        X_\alpha & L^s \\};
      \path[commutative diagrams/.cd, every arrow, every label]
      (m-1-1) edge[commutative diagrams/hook] (m-1-2)
      (m-2-1) edge[commutative diagrams/hook] (m-2-2)
              edge node {$T^{-1}$} (m-1-1)
      (m-2-2) edge node[swap] {$T^{-1}$} (m-1-2);
    \end{tikzpicture}
  \end{equation*}
  is a commuting diagram of linear maps between Banach spaces, where
  the inclusions and the map $T^{-1}\colon L^s\to L^s$ are continuous.
  By the closed graph theorem also $T^{-1}\colon X_\alpha\to X_1$ is
  continuous.  Moreover, if $s>N$ we have continuous maps
  \begin{equation*}
    X_\alpha\hookrightarrow L^s\xrightarrow{T^{-1}}
    W^{2,s}\hookrightarrow C^1_{\rmb}
  \end{equation*}
  so $T^{-1}\colon X_\alpha \to C^1_{\rmb}$ is continuous.  Similarly,
  \begin{equation*}
    X_\alpha\hookrightarrow L^2\xrightarrow{T^{-1}}
    H^2\hookrightarrow H^1
  \end{equation*}
  and therefore $T^{-1}\colon X_\alpha\to H^1$ is continuous.  All in
  all we have proved \ref{item:15}.

  \textbf{\ref{item:16}:} Note that $\cF(u)\in X_q \subseteq X_\alpha$
  if $u\in X_1$, by \eqref{eq:4}.  To see the continuous
  differentiability of $\cF$ in $Y$, note that $f'$ is locally Hölder
  (respectively Lipschitz) continuous in $\dR$ with exponent
  $\beta\coloneqq\min\{1,q-1\}$, as a consequence of \ref{item:4} or
  \ref{item:3}, respectively.  In what follows we repeatedly pick
  arbitrary $u,v,w\in X_1$ and $C>0$ such that
  $\abs{f'(s)-f'(t)}\le C\abs{s-t}^\beta$ for all $s,t\in \dR$ with
  $\abs{s},\abs{t}\le \norm{u}_\infty+ \norm{v}_\infty$.  Define
  $\cF_1$ to be the superposition operator induced by $f'$.  First we
  show that $\cF_1(u)\in\cL(X_1, X_\alpha)$ as a multiplication
  operator and that $\cF_1\colon X_1\to \cL(X_1,X_\alpha)$ is
  continuous.  Pick $a_1$ in \eqref{eq:4} for
  $M\coloneqq\norm{u}_\infty$.  Then we find
  \begin{equation*}
    \norm{\cF_1(u)w}_{X_\alpha}
    \le a_1\norm{\varphi^{q-\alpha}}_\infty\norm{u}_{X_1}^{q-1}\norm{w}_{X_1}
  \end{equation*}
  with $\norm{\varphi^{q-\alpha}}_\infty<\infty$ since $\alpha<q$.
  Hence $\cF_1(u)\in\cL(X_1,X_\alpha)$.  Similarly,
  \begin{equation*}
    \norm{(\cF_1(u)-\cF_1(v))w}_{X_\alpha}
    \le C\norm{\varphi^{\beta+1-\alpha}}_\infty\norm{u-v}_{X_1}^\beta\norm{w}_{X_1}
  \end{equation*}
  with $\norm{\varphi^{\beta+1-\alpha}}_\infty<\infty$ since
  $\alpha<\beta+1$.  Hence
  \begin{equation*}
    \norm{\cF_1(u)-\cF_1(v)}_{\cL(X_1,X_\alpha)}
    \le C\norm{\varphi^{\beta+1-\alpha}}_\infty\norm{u-v}_{X_1}^\beta
  \end{equation*}
  and $\cF_1$ is Hölder continuous.  For any $x\in\dR^N$ and
  $t\in\dR\ssm\{0\}$ there is $\theta_{x,t}\in(-\abs{t},\abs{t})$ such
  that
  \begin{multline*}
    \xabs{\frac{f(u(x)+tv(x))-f(u(x))}{t}-f'(u(x))v(x)}
    =\xabs{f'(u(x)+\theta_{x,t}v(x))-f'(u(x))}\abs{v(x)}\\
    \le C\abs{\theta_{x,t}v(x)}^\beta\abs{v(x)}
    \le C\abs{v(x)}^{\beta+1}\abs{t}^\beta.
  \end{multline*}
  It follows that
  \begin{equation*}
    \xnorm{\frac{\cF(u+tv)-\cF(u)}{t}-\cF_1(u)v}_{X_\alpha}
    \le C\norm{\varphi^{\beta+1-\alpha}}_\infty\norm{v}_{X_1}^{\beta+1}\abs{t}^\beta
  \end{equation*}
  and hence that $\cF$ is Gâteaux differentiable in $u$ with
  derivative $\cF_1(u)$.  Since $\cF_1$ is continuous, $\cF$ is
  continuously Fréchet differentiable as a map $X_1\mapsto X_\alpha$,
  and thus $Y\hookrightarrow X_1$ implies continuous differentiability
  of $\cF\colon Y\to X_\alpha$.

  Suppose now that $(u_n)\subseteq Y$ is bounded in $Y$ and hence
  bounded in $X_1$ and $C^1_{\rmb}(\dR^N)$.  Passing to a subsequence
  we can suppose by Arzelà-Ascoli's theorem that $(u_n)$ converges
  locally uniformly in $\dR^N$ to some $u\in C_\rmb(\dR^N)$.  Since
  $f$ is uniformly continuous on compact intervals, $\cF(u_n)$
  converges to $\cF(u)$ locally uniformly in $\dR^N$.  There is $C>0$
  such that
  \begin{equation}\label{eq:15}
    u,u_n\le C\varphi\qquad\text{in }\dR^N, \text{ for all }n\in\dN.
  \end{equation}
  For any $\varepsilon>0$ \eqref{eq:4} and \eqref{eq:15} imply that
  there are a constant $R>0$, constants $C>0$, and $n_0\in\dN$ such
  that for all $n\ge n_0$ it holds true that
  \begin{align*}
    \frac{\abs{\cF(u_n)}}{\varphi^\alpha} \le C\varphi^{q-\alpha}
    &\le\frac{\varepsilon}{3}\qquad\text{in }\dR^N\ssm B_R,\\
    \frac{\abs{\cF(u)}}{\varphi^\alpha} \le C\varphi^{q-\alpha}
    &\le\frac{\varepsilon}{3}\qquad\text{in }\dR^N\ssm B_R,\\
    \shortintertext{and} \frac{\abs{\cF(u_n)-\cF(u)}}{\varphi^\alpha}
    \le C\norm{\cF(u_n)-\cF(u)}_\infty
    &\le\frac{\varepsilon}{3}\qquad\text{in }\olB_R.
  \end{align*}
  It follows for $n\ge n_0$ that
  \begin{equation*}
    \norm{\cF(u_n)-\cF(u)}_{X_\alpha}
    \le\sup_{\olB_R}\frac{\abs{\cF(u_n)-\cF(u)}}{\varphi^\alpha}
    +\sup_{\dR^N\ssm
      B_R}\frac{\abs{\cF(u_n)}+\abs{\cF(u)}}{\varphi^\alpha}
    \le\varepsilon
  \end{equation*}
  and hence $\cF(u_n)\to\cF(u)$ in $X_\alpha$.  This proves that $\cF$
  maps bounded sets in $Y$ into relatively compact sets in $X_\alpha$.
  Since $\cF$ is differentiable, it is completely continuous.

  \textbf{\ref{item:23}:} Fix $u\in K_+$.  By
  \th\ref{cor:compare-solutions} there is $C_1>0$ such that
  $C_1\varphi\le u$ in $\dR^N$.  For any $v\in Y$ such that
  $\norm{u-v}_Y\le C_1/2$ it follows that $\norm{u-v}_{X_1}\le C_1/2$
  and hence
  \begin{equation*}
    v\ge u-\abs{u-v}\ge\frac{C_1\varphi}{2}>0
  \end{equation*}
  in $\dR^N$.  Therefore, $u$ lies in the interior of the positive
  cone of $Y$.  

  \textbf{\ref{item:17}:} It suffices to prove that on $K$ the
  $H^1$-topology is finer than the $Y$-topology.  Therefore, assume
  that $u_n\to u$ in $K$ with respect to the $H^1$-topology and
  suppose by contradiction that $u_n\not\to u$ in $Y$.  Passing to a
  subsequence we can assume that there is $\delta>0$ such that
  \begin{equation}
    \label{eq:7}
    \norm{u_n-u}_Y\ge\delta\qquad\text{for all }n\in\dN.
  \end{equation}
  By \ref{item:20} and standard elliptic regularity estimates, $(u_n)$
  is bounded in $C^1_\rmb(\dR^N)$.  Moreover, the proof of
  \cite[Prop.~5.2]{MR2151860} yields, together with regularity
  estimates, that the functions $u_n$ have a uniform pointwise
  exponential decay as $\abs{x}\to\infty$.  In view of \eqref{eq:4} we
  obtain $C_1,C_2>0$ such that
  \begin{equation*}
    \frac{f(u_n(x))}{u_n(x)}\le C_1\rme^{-C_2\abs{x}}
    \qquad\text{for } x\in\dR^N,\ n\in\dN. 
  \end{equation*}
  By \th\ref{thm:decay-at-infty}\ref{item:12} $(u_n)$ also remains
  bounded in $X_1$ and hence in $Y$.  Pick some
  $\alpha\in(1,\min\{2,q\})$.  By \ref{item:15} and \ref{item:16}, and
  passing to a subsequence, $(T^{-1}\cF(u_n))$ converges in $Y$.
  Since $u_n= T^{-1}\cF(u_n)$, $u_n\to v$ in $Y$, for some $v\in Y$,
  and $v=u$ since $Y\hookrightarrow H^1$ and $u_n\to u$ in $H^1$.
  Hence $u_n\to u$ in $Y$ for this subsequence, contradicting
  \eqref{eq:7} and thus finishing the proof of \ref{item:17}.
\end{proof}

If \ref{item:2} is satisfied then $J$, as defined in the introduction,
is well defined on $Y$.  The main result of this section is the
following
\begin{theorem}
  \th\label{thm:analyticity} Assume that \ref{item:1} and \ref{item:2}
  hold true.  
  \begin{enumerate}[label=\textup{(\alph*)}]
  \item \label{item:10} If \ref{item:4} is satisfied then $K$ is
    $Y$-locally path connected, and $J$ is $Y$-locally constant on $K$.
  \item \label{item:11} If \ref{item:3} is satisfied then $K_+$ is
    $Y$-locally path connected, and $J$ is $Y$-locally constant on $K_+$.
  \end{enumerate}
\end{theorem}
\begin{proof}
  We prove the two statements in parallel.  Fix
  $\alpha\in(1,\min\{2,q\})$, where $q$ is taken from \eqref{eq:4}.
  For \ref{item:10} fix $u\in K$, and for \ref{item:11} fix $u\in
  K_+$.  Set $M\coloneqq\norm{u}_\infty$ and let the numbers $a_k$ be
  given by \ref{item:4} or \ref{item:3}, respectively.
  
  Denote by $\cL^k(X_1,X_\alpha)$ the Banach space of $k$-linear
  bounded maps from $X_1$ into $X_\alpha$, for $k\in\dN_0$ (for $k=0$
  we set $\cL^k(X_1,X_\alpha)\coloneqq X_\alpha$).  For $k=0$ and
  $k=1$ we already know that $f^{(k)}(u)$ generates an element of
  $\cL^k(X_1,X_\alpha)$ by multiplication by
  \th\ref{lem:properties-y}\ref{item:16}.  We claim that
  \begin{gather}
    \label{eq:11}
    \parbox{.8\linewidth}{$f^{k}(u)$ generates an element $A_k$ of
      $\cL^k(X_1,X_\alpha)$ by multiplication, for every $k\in\dN_0$,}\\
    \label{eq:26}
    r_1\coloneqq\xlr(){\limsup_{k\to\infty}\norm{A_k}_{\cL^k(X_1,X_\alpha)}^{1/k}}^{-1}
    >0,\\
    \shortintertext{and}
    \exists r_2\in(0,r_1]\ \forall h\in B_{r_2}X_1\ \forall x\in\dR^N\colon
    f(u(x)+h(x))=\sum_{k=0}^\infty\frac{f^{(k)}(u(x))}{k!}h(x)^k.\label{eq:8}
  \end{gather}

  To prove the claims in case \ref{item:10}, denote by $r_0$ the
  convergence radius of the power series
  $\sum_0^\infty\frac{a_k}{k!}z^k$.  Consider $k\in\dN$, $k\ge2$.
  Taking into account that $\alpha<2$ we obtain from \ref{item:4} that
  \begin{equation*}
    \xnorm{\frac{f^{(k)}(u)}{k!}h^k}_{X_\alpha}
    \le\frac{a_k}{k!} \sup_{x\in\dR^N}
    \xabs{\frac{h(x)^k}{\varphi(x)^\alpha}}
    \le \frac{a_k}{k!}\norm{\varphi}_\infty^{k-\alpha} \norm{h}_{X_1}^k.
  \end{equation*}
  Hence \eqref{eq:11} is true, with
  \begin{equation*}
    \norm{A_k}_{\cL^k(X_1,X_\alpha)}
    \le \frac{a_k}{k!}\norm{\varphi}_\infty^{k-\alpha}.
  \end{equation*}
  Again by \ref{item:4}, \eqref{eq:26} is satisfied, and
  \begin{equation*}
    r_2\coloneqq\frac{r_0}{\norm{\varphi}_\infty}\le r_1.
  \end{equation*}
  Suppose now that $h\in B_{r_2}X_1$ and $x\in\dR^N$.  Then
  $u(x)\in[-M,M]$ and hence by \ref{item:4}
  \begin{equation*}
    \xlr(){\limsup_{k\to\infty}\fracwithdelims\lvert\rvert{f^{(k)}(u(x))}{k!}}^{-1}
    \ge r_0.
  \end{equation*}
  Moreover, $\abs{h(x)}< r_2\varphi(x)\le r_0$.  Since $f$ is
  analytic, \eqref{eq:8} follows.

  To prove the claims in case \ref{item:11}, denote again by $r_0$ the
  convergence radius of the power series
  $\sum_0^\infty\frac{a_k}{k!}z^k$.  By
  \th\ref{cor:compare-solutions} there are $C_1,C_2>0$ such that
  \begin{equation}
    \label{eq:6}
    C_1\varphi\le u\le C_2\varphi. 
  \end{equation}
  Suppose first that $k\in\dN$, $2\le k\le p$.  Taking into account that
  $\alpha<q$ we obtain from \ref{item:3}
  \begin{equation*}
    \xnorm{\frac{f^{(k)}(u)}{k!}h^k}_{X_\alpha}
    \le\frac{a_k}{k!} \sup_{x\in\dR^N}
    \abs{u(x)}^{q-k}
    \xabs{\frac{h(x)^k}{\varphi(x)^\alpha}}
    \le \frac{a_k}{k!}\norm{\varphi^{q-\alpha}}_\infty C_2^{q-k}\norm{h}_{X_1}^k.
  \end{equation*}
  If $k>q$ then we find
  \begin{equation*}
    \xnorm{\frac{f^{(k)}(u)}{k!}h^k}_{X_\alpha}
    \le\frac{a_k}{k!} \sup_{x\in\dR^N}
    \abs{u(x)}^{q-k}
    \xabs{\frac{h(x)^k}{\varphi(x)^\alpha}}
    \le \frac{a_k}{k!}\norm{\varphi^{q-\alpha}}_\infty C_1^{q-k}\norm{h}_{X_1}^k.
  \end{equation*}
  Hence \eqref{eq:11} is true, with
  \begin{equation*}
    \norm{A_k}_{\cL^k(X_1,X_\alpha)}\le
    \frac{a_k}{k!}\norm{\varphi^{p-\alpha}}_\infty C_1^{q-k}
  \end{equation*}
  for $k>p$.  Again by \ref{item:3}, \eqref{eq:26} is satisfied, and
  \begin{equation*}
    r_2\coloneqq C_1\min\xlr\{\}{r_0,\frac12}\le r_1.
  \end{equation*}

  Suppose now that $h\in B_{r_2}X_1$ and $x\in\dR^N$.  Then
  $u(x)\in(0,M]$ and hence by \ref{item:3}
  \begin{equation*}
    \xlr(){\limsup_{k\to\infty}\fracwithdelims\lvert\rvert{f^{(k)}(u(x))}{k!}}^{-1}
    \ge r_0u(x).
  \end{equation*}
  Moreover, $\abs{h(x)}< r_2\varphi(x)\le C_1r_0\varphi(x)\le
  r_0u(x)$, by \eqref{eq:6}, and hence $u(x)+h(x)\ge
  (C_1-r_2)\varphi(x)>0$.  Since $f$ is analytic in $(0,\infty)$,
  \eqref{eq:8} follows.

  For any $h\in X_1$ such that $\norm{h}_{X_1}<r_2$ we obtain from
  \eqref{eq:26} and $r_2\le r_1$ that
  \begin{equation}
    \label{eq:16}
    \sum_{k=0}^\infty A_k[h^k]\qquad\text{converges in $X_\alpha$.}
  \end{equation}
  Note that $X_\alpha$ embeds continuously in $C_\rmb(\dR^N)$ and that
  therefore the evaluation $E_x$ at a point $x\in\dR^N$ is a bounded
  linear operator on $X_\alpha$.
  Hence for every $x\in\dR^N$
  \begin{align*}
    \cF(u+h)(x)
    &=f(u(x)+h(x))\\
    &=\sum_{k=0}^\infty\frac{f^{(k)}(u(x))}{k!}h(x)^k&&\text{by \eqref{eq:8}}\\
    &=\sum_{k=0}^\infty E_x\xlr[]{A_k[h^k]}&&\text{by \eqref{eq:11}}\\
    &=E_x\xlr[]{\sum_{k=0}^\infty A_k[h^k]}&&\text{by \eqref{eq:16} and }E_x\in\cL(X_\alpha,\dR)
  \end{align*}
  and therefore
  \begin{equation}
    \cF(u+h)=\sum_{k=0}^\infty A_k[h^k],
    \qquad \text{for all } h\in B_{r_2}X_1.
  \end{equation}
  By \cite[Theorem~6.2]{MR0062947} the map $\cF$ is analytic in
  $B_{r_2}X_1$.  Since $u\in K_{(+)}$ was arbitrary, $\cF\colon X_1\to
  X_\alpha$ is analytic in a neighborhood of $K_{(+)}$.  And since
  $Y\hookrightarrow X_1$ and bounded linear operators are analytic,
  also $\cF\colon Y\to X_\alpha$ is analytic in a neighborhood of
  $K_{(+)}$, c.f.~\cite[Theorem~7.3]{MR0313811}.

  From the results above we conclude that $\Gamma\colon Y\to Y$ is
  analytic in a neighborhood of $K_{(+)}$.  Moreover, by
  \th\ref{lem:properties-y}\ref{item:15} $\Gamma$ is continuously
  differentiable in $Y$, and by \th\ref{lem:properties-y}\ref{item:16}
  and \cite[Proposition~8.2]{MR787404}, for every $v\in Y$ the
  operator $\cF'(v)\in\cL(Y,X_\alpha)$ is compact.  Hence for every
  $v\in Y$ the operator $\Gamma'(v)$ is of the form identity minus
  compact and thus a Fredholm operator of index $0$.  In short, one
  calls the map $\Gamma$ a Fredholm map of index $0$.

  Recall that $K=\Gamma^{-1}(0)$.  By
  \th\ref{lem:properties-y}\ref{item:23} $K_+$ is the set of zeros of
  $\Gamma$ in an open neighborhood of $u$.  In any case, the implicit
  function theorem shows that there are an open neighborhood $U$ of
  $u$ in $Y$ and a $C^\infty$-manifold $M\subseteq Y$ of finite
  dimension $\dim\cN(\Gamma'(u))$ such that $K\cap U\subseteq M$.  In
  fact, by \cite[Theorem~7.5]{MR0313811} (see also Corollary~7.3
  \emph{loc.\ cit.}), $M$ is the graph of a analytic map defined on a
  neighborhood of $u$ in $u+\cN(\Gamma'(u))$.  Moreover, $K\cap U$ is
  the set of zeros of the restriction of the finite dimensional
  analytic map $P\Gamma$ to $M$.  Here $P\in\cL(Y)$ denotes the
  projection with kernel $\cR(\Gamma'(u))$ and range
  $\cN(\Gamma'(u))$.  Therefore, \cite[Theorem~2]{MR0173265} applies
  and yields a triangulation of $K\cap U$ by homeomorphic images of
  simplexes such that their interior is mapped analytically (see also
  \cite[Satz~4]{MR0159346}).  This implies that $K_{(+)}$ is locally
  path connected by piecewise continuously differentiable arcs.
  Similarly as in the proof of \th\ref{lem:properties-y} it can be
  shown that the map $Y\to\dR$, $u\mapsto \int F(u)$ is continuously
  differentiable.  Hence also $J$ is continuously differentiable in
  $Y$ and therefore locally constant on $K_{(+)}$.
\end{proof}

\section{Applications to Periodic Potentials}
\label{sec:appl-peri-potent}

Returning to our main motivation we consider the variational setting
in $H^1(\dR^N)$.  Assuming \ref{item:1}, \ref{item:2}, and
\ref{item:20} the functional $J$ is of class $C^1$ on $H^1(\dR^N)$,
and solutions of \eqref{eq:12} are in correspondence with critical
points of $J$.  Denoting $c_0\coloneqq\inf J(K)$ it is easy to see
that $c_0>0$ if $K\neq\varnothing$.

To inspect the behavior of $J$ on $K$ we will need the following
boundedness condition:
\begin{enumerate}[fcond]
\item \label{item:21} Every sequence $(u_n)\subseteq K$ such that
  $\limsup_{n\to\infty} J(u_n)<2c_0$ is bounded.
\end{enumerate}
It is satisfied, for example, under the classical
Ambrosetti-Rabinowitz condition.  Alternatively, one could use a set
of conditions as in \cite{MR2557725}.  

For our purpose we also consider the periodicity condition
\begin{enumerate}[vcond]
\item \label{item:18} $V$ is $1$-periodic in all coordinates.
\end{enumerate}
By concentration compactness arguments $c_0$ is achieved if
\ref{item:1}, \ref{item:18}, \ref{item:2}, \ref{item:20} and
\ref{item:21} hold true and if $K\neq\varnothing$.

The local path connectedness of the set of (positive) solutions of
\eqref{eq:12} when $f$ is analytic has a consequence on the possible
critical levels of $J$:
\begin{theorem}
  \th\label{thm:discrete-energy-levels} Assume \ref{item:1},
  \ref{item:18}, \ref{item:2}, \ref{item:20} and \ref{item:21}.
  \begin{enumerate}[label=\textup{(\alph*)}]
  \item \label{item:19} If \ref{item:4} is satisfied, then
    $J(K)$ has no accumulation point in $[c_0,2c_0)$.
  \item \label{item:22} If  \ref{item:3} is satisfied, then
    $J(K_+)$ has no accumulation point in $[c_0,2c_0)$.
  \end{enumerate}
\end{theorem}
\begin{proof} 
  We only prove \ref{item:19} since the other claim is proved
  analogously.  Assume by contradiction that $J(K)$ contains an
  accumulation point $c\in[c_0,2c_0)$.  We work entirely in the
  $H^1$-topology, which coincides with the $Y$-topology on $K$ by
  \th\ref{lem:properties-y}\ref{item:17}.  There is a sequence
  $(u_n)\subseteq K$ such that $J(u_n)\neq c$ and $J(u_n)\to c$.  A
  standard argument using the splitting lemma
  \cite[Proposition~2.5]{MR2151860} yields, after passing to a
  subsequence, a translated sequence $(v_n)\subseteq K$ and $v\in K$
  such that $v_n\to v$, $J(v_n)=J(u_n)\neq c$ and $J(v)=c$.  Since $J$
  is locally constant on $K$ by \th\ref{thm:analyticity}\ref{item:10}
  we reach a contradiction.
\end{proof}

We now combine this property with the separation property obtained
in \cite{MR2488693} to show the existence of compact isolated sets
of solutions.  For any $c\in\dR$ denote
\begin{equation*}
  K_+^c\coloneqq\{u\in K_+\mid J(u)\le c\}.
\end{equation*}
The result reads:
\begin{corollary}
  \th\label{cor:isolated-set} In the situation of
  \th\ref{thm:discrete-energy-levels}\ref{item:22}, assume in
  addition that $V$ is of class $C^{1,1}$, that $V$ is even in every
  coordinate $x^i$, and that there is $\theta>2$ such that
  \begin{equation*}
    f'(u)u^2\ge(\theta-1)f(u)u\qquad\text{for }u\in\dR\ssm\{0\}.
  \end{equation*}
  Suppose that for every $u\in K_+^{c_0}$ that is even in $x^i$ for some
  $i\in\{1,2,\dots,N\}$ it holds true that
  \begin{equation*}
    \int_{\dR^N}u^2\partial_i^2 V\le0.
  \end{equation*}
  (Here we use the weak second derivative of $V$.  It exists because
  $V'$ is Lipschitz continuous.)  Then $K^+\neq\varnothing$ and there
  exists a compact subset $\Lambda$ of $K_+^{c_0}$ that is isolated in
  $K$, i.e., that satisfies $\dist(\Lambda,K\ssm\Lambda)>0$ in the
  $H^1$-metric.
\end{corollary}
\begin{proof}
  By \cite[Theorem~1.1]{MR2488693} there is a compact subset $\Lambda$ of
  $K_+^{c_0}$ such that
  \begin{equation*}
    K_+^{c_0}=\dZ^N\star\Lambda
    \qquad\text{and}\qquad
    \Lambda\cap\xlr(){\dZ^N\ssm\{0\}}\star\Lambda=\varnothing.
  \end{equation*}
  Here $\star$ denotes the action of $\dZ^N$ on functions on $\dR^N$
  by translation: $a\star u\coloneqq u(\cdot\,-a)$.  It follows
  easily that
  \begin{equation}
    \label{eq:17}
    \dist(\Lambda,K_+^{c_0}\ssm\Lambda)>0.
  \end{equation}

  We claim that $\dist(\Lambda,K\ssm\Lambda)>0$.  Recall that the
  topologies of the space $Y$ from Section~\ref{sec:real-analyticity}
  and the $H^1$-topology coincide on $K$ and that $K_+$ is contained
  in the interior of the positive cone of $Y$, by
  \th\ref{lem:properties-y}\ref{item:17} and \ref{item:23}.  Hence
  $\dist(\Lambda,K\ssm K_+)>0$.  It remains to show that
  $\dist(\Lambda,K_+\ssm\Lambda)>0$.  Assume by contradiction that
  this were not the case.  Since $\Lambda$ is compact there would
  exist a sequence $(u_n)\subseteq K_+\ssm\Lambda$ and $u\in\Lambda$
  such that $u_n\to u$.  Since $c_0$ is not an accumulation point of
  $J(K_+)$ by \th\ref{thm:discrete-energy-levels}\ref{item:22},
  $(u_n)\subseteq K_+^{c_0}$.  But this contradicts \eqref{eq:17},
  proving the claim.
\end{proof}

Note that in \cite{MR2488693} we show how to construct concrete
examples that satisfy the conditions of \th\ref{cor:isolated-set}.

\begin{acknowledgement}
  We would like to thank Jawad Snoussi for drawing our attention to
  the references \cite{MR0159346,MR0173265}.
\end{acknowledgement}
\bibliographystyle{amsplain-abbrv}
\bibliography{template}

\def\cprime{$'$} \def\polhk#1{\setbox0=\hbox{#1}{\ooalign{\hidewidth
  \lower1.5ex\hbox{`}\hidewidth\crcr\unhbox0}}}
\providecommand{\bysame}{\leavevmode\hbox to3em{\hrulefill}\thinspace}
\providecommand{\MR}{\relax\ifhmode\unskip\space\fi MR }
\providecommand{\MRhref}[2]{%
  \href{http://www.ams.org/mathscinet-getitem?mr=#1}{#2}
}
\providecommand{\href}[2]{#2}
\begin{thebibliography}{10}

\bibitem{MR2488693}
N.~Ackermann, \emph{Solution set splitting at low energy levels in
  {S}chr{\"o}dinger equations with periodic and symmetric potential}, J.
  Differential Equations \textbf{246} (2009), no.~4, 1470--1499. \MR{MR2488693}

\bibitem{MR2151860}
N.~Ackermann and T.~Weth, \emph{Multibump solutions of nonlinear periodic
  {S}chr{\"o}dinger equations in a degenerate setting}, Commun. Contemp. Math.
  \textbf{7} (2005), no.~3, 269--298. \MR{MR2151860}

\bibitem{MR0062947}
A.~Alexiewicz and W.~Orlicz, \emph{Analytic operations in real {B}anach
  spaces}, Studia Math. \textbf{14} (1953), 57--78 (1954). \MR{0062947
  (16,47d)}

\bibitem{MR890161}
A.~Ancona, \emph{Negatively curved manifolds, elliptic operators, and the
  {M}artin boundary}, Ann. of Math. (2) \textbf{125} (1987), no.~3, 495--536.
  \MR{890161 (88k:58160)}

\bibitem{MR1482989}
\bysame, \emph{First eigenvalues and comparison of {G}reen's functions for
  elliptic operators on manifolds or domains}, J. Anal. Math. \textbf{72}
  (1997), 45--92. \MR{MR1482989 (98i:58212)}

\bibitem{MR0313811}
J.~Bochnak and J.~Siciak, \emph{Analytic functions in topological vector
  spaces}, Studia Math. \textbf{39} (1971), 77--112. \MR{0313811 (47 \#2365)}

\bibitem{MR93k:35087}
V.~Coti~Zelati and P.H. Rabinowitz, \emph{Homoclinic type solutions for a
  semilinear elliptic {PDE} on {$\mathbf{R}\sp n$}}, Comm. Pure Appl. Math.
  \textbf{45} (1992), no.~10, 1217--1269. \MR{93k:35087}

\bibitem{MR0322615}
E.N. Dancer, \emph{Bifurcation theory for analytic operators}, Proc. London
  Math. Soc. (3) \textbf{26} (1973), 359--384. \MR{0322615 (48 \#977)}

\bibitem{MR0375019}
\bysame, \emph{Global structure of the solutions of non-linear real analytic
  eigenvalue problems}, Proc. London Math. Soc. (3) \textbf{27} (1973),
  747--765. \MR{0375019 (51 \#11215)}

\bibitem{MR1962054}
\bysame, \emph{Real analyticity and non-degeneracy}, Math. Ann. \textbf{325}
  (2003), no.~2, 369--392. \MR{1 962 054}

\bibitem{MR787404}
K.~Deimling, \emph{Nonlinear functional analysis}, Springer-Verlag, Berlin,
  1985. \MR{MR787404 (86j:47001)}

\bibitem{MR0159346}
B.~Giesecke, \emph{Simpliziale {Z}erlegung abz\"ahlbarer analytischer
  {R}\"aume}, Math. Z. \textbf{83} (1964), 177--213. \MR{0159346 (28 \#2563)}

\bibitem{MR737190}
D.~Gilbarg and N.S. Trudinger, \emph{Elliptic partial differential equations of
  second order}, second ed., Grundlehren der Mathematischen Wissenschaften
  [Fundamental Principles of Mathematical Sciences], vol. 224, Springer-Verlag,
  Berlin, 1983. \MR{MR737190 (86c:35035)}

\bibitem{MR0173265}
S.~Lojasiewicz, \emph{Triangulation of semi-analytic sets}, Ann. Scuola Norm.
  Sup. Pisa (3) \textbf{18} (1964), 449--474. \MR{0173265 (30 \#3478)}

\bibitem{MR874676}
M.~Murata, \emph{Structure of positive solutions to {$(-\Delta+V)u=0$} in
  {$\mathbf{R}\sp n$}}, Duke Math. J. \textbf{53} (1986), no.~4, 869--943.
  \MR{MR874676 (88f:35039)}

\bibitem{MR1326606}
R.G. Pinsky, \emph{Positive harmonic functions and diffusion}, Cambridge
  Studies in Advanced Mathematics, vol.~45, Cambridge University Press,
  Cambridge, 1995. \MR{1326606 (96m:60179)}

\bibitem{MR2001f:35139}
P.J. Rabier and C.A. Stuart, \emph{Exponential decay of the solutions of
  quasilinear second-order equations and {P}ohozaev identities}, J.
  Differential Equations \textbf{165} (2000), no.~1, 199--234. \MR{2001f:35139}

\bibitem{MR2557725}
A.~Szulkin and T.~Weth, \emph{Ground state solutions for some indefinite
  variational problems}, J. Funct. Anal. \textbf{257} (2009), no.~12,
  3802--3822. \MR{2557725 (2010j:35193)}

\end{thebibliography}
\subsubsection*{Contact information:}
\begin{sloppypar}
  \begin{description}
  \item[Nils Ackermann:] Instituto de Matem\'{a}ticas, Universidad
    Nacional Aut\'{o}noma de M\'{e}xico, Circuito Exterior, C.U.,
    04510 M\'{e}xico D.F., Mexico
  \item[Norman Dancer:] School of Mathematics and Statistics,
    University of Sydney, NSW 2006, Australia
  \end{description}
\end{sloppypar}

\end{document}